\documentclass{amsart}
\usepackage[foot]{amsaddr}
\usepackage[T1]{fontenc}
\usepackage[utf8]{inputenc}
\usepackage{geometry}
\usepackage{amsmath}
\usepackage{amsthm}
\usepackage{bm}
\usepackage{comment}
\usepackage[colorlinks=true, allcolors=blue]{hyperref}
\usepackage[capitalize, nameinlink]{cleveref}

\usepackage{todonotes}
\makeatletter
\theoremstyle{plain}
\newtheorem{thm}{\protect\theoremname}
\theoremstyle{plain}
\newtheorem{lem}[thm]{\protect\lemmaname}
\theoremstyle{plain}
\newtheorem{definition}[thm]{Definition}
\newtheorem{proposition}[thm]{Proposition}
\newtheorem{remark}[thm]{Remark}

\newcommand{\mb}{\mathbb}

\newcommand{\mc}{\mathcal}

\newcommand{\on}{\operatorname}

\newcommand{\wt}{\widetilde}

\renewcommand{\epsilon}{\varepsilon}

\usepackage{amssymb}
\usepackage{enumerate}

\makeatother

\providecommand{\lemmaname}{Lemma}
\providecommand{\theoremname}{Theorem}

\begin{document}
\title{Optimal thresholds for Latin squares, Steiner Triple Systems, and edge colorings}
\author{Vishesh Jain}
\address{Department of Mathematics, Statistics, and Computer Science, University of Illinois Chicago, Chicago IL USA.}
\author{Huy Tuan Pham}
\address{Department of Mathematics, Stanford University, Stanford CA USA.}
\email[A1]{visheshj@uic.edu}
\email[A2]{huypham@stanford.edu}
\maketitle

\begin{abstract}
We show that the threshold for the binomial random $3$-partite, $3$-uniform hypergraph $G^{3}((n,n,n),p)$ to contain a Latin square is $\Theta(\log{n}/n)$. We also prove analogous results for Steiner triple systems and proper list edge-colorings of the complete (bipartite) graph with random lists. Our results answer several related questions of Johansson, Luria-Simkin, Casselgren-H\"aggkvist, Simkin, and Kang-Kelly-K\"uhn-Methuku-Osthus.       
\end{abstract}

\section{Introduction}
Given a finite set $X$ and $p \in (0,1)$, let $X_{p} \subseteq X$ be a random subset where each element of $X$ is sampled independently with probability $p$. For
a non-trivial monotone property ${\mathcal P}\subseteq2^{X}$ over subsets of $X$, the threshold
$p_{c}({\mathcal P})$ of ${\mathcal P}$ is the value $p^*$ at which $\mathbb{P}(X_{p^*}\in{\mathcal P})=1/2$.
When $X=\binom{[n]}{k}$, $X_{p}$ is precisely the Erd\H{o}s-R\'enyi
random $k$-uniform hypergraph $G^{(k)}(n,p)$, for which estimating thresholds
of interesting combinatorial properties has been a major direction in combinatorics, going back to a seminal result of Erd\H{o}s and R\'enyi that the threshold for the appearance of perfect matchings in $G^{(2)}(n,p)$ is $\log{n}/n$ \cite{ER}. The determination of the threshold for the appearance of perfect matchings in $k$-uniform hypergraphs for $k\geq 3$ is a notorious problem of Shamir, for which the threshold $\Theta((\log n)/n^{k-1})$ was established in the celebrated work of Johansson, Kahn
and Vu \cite{JKV}. 

Recently, Park and the second author \cite{PP} gave
a proof of the Kahn--Kalai conjecture (a weaker fractional version conjectured by Talagrand \cite{Tal} was obtained earlier in work of Frankston, Kahn, Narayanan, and Park \cite{FKNP} building on the sunflower breakthrough of Alweiss, Lovett, Wu and Zhang \cite{ALWZ}), which allows one (among other things) to get a much simpler
proof of the main result of \cite{JKV}. To state the result of \cite{PP}, we say that ${\mathcal H}\subseteq2^{X}$
is {$p$-small} if there exists ${\mathcal G}\subseteq2^{X}$
such that any set in ${\mathcal H}$ contains a set from ${\mathcal G}$
and $\sum_{G\in{\mathcal G}}p^{|G|}<\frac{1}{2}$. 
Moreover, a distribution $\mu$ supported on $\mc{H}$ is said to be $p$-spread if for all $S\subseteq X$, $\mu(\{W\in \mc{H}:S\subseteq W\})\le 2 p^{|S|}$. 
Using linear programming duality, Talagrand
\cite{Tal} observed that if ${\mathcal H}$ supports a $p$-spread distribution $\mu$, then $\mc{H}$ is not $p$-small. 
\begin{thm}[\cite{PP}]
\label{thm:FracKK}There exists $C>0$ such that the following holds.
Let ${\mathcal H}\subseteq2^{X}$ and $p \in (0,1)$ be such that ${\mathcal H}$ is not $p$-small. Then, $X_{Cp\log|X|}$ contains a set in ${\mathcal H}$ with high probability (i.e.~probability going to $0$ as $|X| \to \infty$).  
In particular, the same conclusion holds if there exists a distribution
$\mu$ supported on ${\mathcal H}$ which is $p$-spread. 
\end{thm}

The study of combinatorial designs has a rich history, dating back to work of Euler in the 18th century and Kirkman in the 19th century. 
A $t$-$(n,k,\lambda)$
design is a $k$-uniform hypergraph on $[n]$ where any subset of
$t$ vertices appears in exactly $\lambda$ hyperedges. Note that a perfect matching in a $k$-uniform hypergraphs is precisely a $1$-$(n,k,1)$ design. The existence of designs for all $t,k,\lambda$ and for all $n \geq n_0(t,k,\lambda)$ satisfying certain necessary divisibility conditions was shown in a landmark work of Keevash \cite{Kee} (see also the work of Glock, K\"uhn, Lo and Osthus \cite{GKLO} for an alternate proof). A Steiner triple system is a $2$-$(n,3,1)$ design. Steiner triple systems are closely related to Latin squares, which are assignments of $[n]$ to an $n\times n$ grid so that each row and each column contains all distinct symbols. Latin
squares can be equivalently described as a tripartite $3$-uniform
hypergraph with vertex parts of size $n$, where each pair of vertices
in different parts is contained in exactly one hyperedge. 

The question of determining the threshold for which $G^{3}((n,n,n),p)$ -- the random tripartite $3$-uniform hypergraph on parts of size $n$ each --  contains a Latin square was raised by Johansson in 2006 \cite{J} and has been popularized by Simkin in the past few years as his favorite open problem. Luria and Simkin \cite{LS} conjectured that this threshold is $\Theta(\log{n}/n)$. It is easily seen (see, e.g.,~\cite{SSS}) that the threshold is $\Omega(\log{n}/n)$, so that the challenge is in proving the upper bound. We mention three related conjectures. Simkin \cite{Sim} conjectured that the
threshold for containment of a Steiner triple system in the random $3$-uniform hypergraph $G^{3}(n,p)$, with $n \equiv 1,3 \mod 6$, is $\Theta((\log n)/n)$,
and more generally, that the same holds for $t$-$(n,t+1,1)$ designs for any fixed $t$ and $n$ satisfying the corresponding divisibility conditions. Given a graph $G$, a random $(k,n)$-list assignment $L$ for edges
of $G$ is an assignment of an independent, uniformly random set $L(e) \in \binom{[n]}{k}$ of colors to each edge $e$ and a proper $L$-list coloring of $G$ is a proper edge-coloring where the color of an edge $e$ belongs to $L(e)$. Casselgren and H\"aggkvist \cite{CH-1} conjectured that for a random
$(O(\log n),n)$-list assignment $L$ for edges of $K_{n,n}$, there is an $L$-list coloring of $K_{n,n}$ with probability at least $1/2$. A non-partite version of this, namely that a random $O((\log{n}), 2n-1)$-list assignment $L$ for edges of $K_{2n}$ admits an $L$-list coloring of $K_{2n}$ with probability at least $1/2$, was conjectured by Kang, Kelly, K\"uhn, Methuku and Osthus \cite{KKKMO}. We refer the reader to \cite{SSS, KKKMO} for further discussion of the history, as well as for more precise versions of the aforementioned conjectures. 

Given \cref{thm:FracKK}, a natural approach to all of these conjectures is to show that the corresponding property is not $O(1/n)$-small, potentially by establishing the stronger statement that the relevant collection of subsets supports an $O(1/n)$-spread distribution. For Shamir's problem, simple counting shows that the uniform distribution on perfect matchings of a $k$-uniform hypergraph on $n$ vertices is $O_k(1/n)$-spread. On the other hand, for complicated structures such as Steiner triple systems and Latin squares, known enumeration results are unfortunately not precise enough to imply anything non-trivial about the spread of the uniform distribution (see the discussion in \cite{KKKMO,SSS,FKNP}). 
In a recent breakthrough work, Sah, Sawhney
and Simkin \cite{SSS} demonstrated the existence of an $O(n^{o(1)}/n)$-spread distribution on Steiner triple systems in $K^{(3)}_{n}$ ($n \equiv 1,3\mod 6$) and Latin squares in $K^{(3)}_{n,n,n}$, thereby establishing the corresponding conjectures about the threshold within a subpolynomial factor $n^{o(1)}$. Their proof constructs such a spread distribution using a clever ``spread boosting'' argument utilizing the iterative absorption framework of K\"uhn, Osthus, and collaborators.   
In a beautiful work using a different and simpler iterative absorption scheme, Kang, Kelly, K\"uhn, Methuku and Osthus \cite{KKKMO} improved the spread parameter to $O(\log{n}/n)$, which is a single logarithmic factor off from the conjectured optimal bound.

In this paper, we show how to construct $O(1/n)$-spread distributions on Steiner triple systems in $K^{(3)}_{n}$ ($n\equiv 1,3 \mod 6$), Latin squares in $K^{(3)}_{(n,n,n)}$, and ordered $1$-factorizations of $K_{2n}$ (i.e.~a decomposition of the edges of $K_{2n}$ into a tuple $(M_1,\dots, M_{2n-1})$ of disjoint perfect matchings), thereby settling the aforementioned conjectures of Johansson, Simkin and Luria, Casselgren and H\"aggkvist, Kang, Kelly, K\"uhn, Methuku and Osthus, and the conjecture of Simkin for the special case of Steiner triple systems.   
\begin{thm}
\label{thm:threshold-designs}
There exists an absolute constant $C>0$ such that for all $n\geq C$, each of the following statements hold with probability at least $1/2$. 
\begin{enumerate}
\item The random $3$-uniform hypergraph $G^{(3)}(n,C(\log n)/n)$
contains a Steiner triple system, provided $n\equiv 1,3 \mod 6$. 
\item The random tripartite $3$-uniform $G^{(3)}((n,n,n),C(\log n)/n)$ contains a Latin square. 
\item A random $(C\log n,n)$-list assignment $L$ for edges of $K_{n,n}$ admits a proper $L$-list coloring. 
\item A random $(C\log n),n)$-list assignment $L$ for edges of $K_{2n}$ admits a proper $L$-list coloring. 
\end{enumerate}
\end{thm}

By combining \cref{thm:FracKK} and the reductions \cite[Theorems~1.6,1.7]{KKKMO} of Kang, Kelly, K\"uhn, Methuku, and Osthus, \cref{thm:threshold-designs} is a direct corollary of the following stronger theorem applied to $K_{n,n}$ (which is trivially $0$-nice, see \cref{def:nice}).

\begin{thm}
\label{thm:factorization-spread}
For any $\epsilon \in (0,1/2)$, there exist sufficiently large constants $N_0 \in \mb{N}$ and $S \in 2^{\mb{N}}-1$ for which the following holds. Let ${\mathbb G} = (A,B,E)$ be a bipartite graph with $|A|=|B|=n$ which is $D_0$-regular for $D_0 \geq \epsilon n \geq N_0$ and $0$-nice (\cref{def:nice}). For $r \in \mb{N}$, let $D_{r} := D_0/S^{r}$.  
Then, for all $r \in \mb{Z}_{\geq 0}$ such that $D_{r} \geq N_0$, there exists a probability distribution $\bm{\mc{P}}_r$ on  decompositions $\mc{P}_r$ of $E(\mb{G})$ with spread $2\epsilon^{-1}\cdot D_r/n$. 
\end{thm}

\subsection{Techniques}
\label{sec:outline}
In \cite{KKKMO}, the authors devised an iterative edge-absorption scheme to construct an $O(\log{n}/n)$-spread distribution on decompositions of a regular, nearly-complete, bipartite graph into regular subgraphs of degree $O(\log{n})$. Briefly, using standard concentration techniques, they first (essentially uniformly) decompose the graph into nearly-regular subgraphs of average degree $\Theta(\log{n})$ satisfying some additional expansion properties, and then iteratively correct these subgraphs to be regular by combining a novel and elegant iterative edge-absorption procedure with a classical result about the existence of (large) regular subgraphs of a given graph (\cref{lem:max-flow}). The initial decomposition is clearly spread, and the key in \cite{KKKMO} is to show that the edge-absorption procedure approximately preserves the spreadness.  
At a very high level, our approach is to recursively use a simpler version of the framework of \cite{KKKMO}, which does not use iterative absorption, until the regular subgraphs in the decomposition have degrees $O(1)$. This requires a number of new ideas, which we now discuss.

\paragraph{\bf The LLL distribution} The reason why the argument of \cite{KKKMO} stops at subgraphs of degree $\Theta(\log{n})$ is that for sub-logarithmic degrees, one cannot use the union bound to guarantee near-regularity of the initial decomposition. However, if one could replace the argument of \cite{KKKMO} by an iterative procedure, where the degree of the new subgraphs is only a constant factor (say) smaller than the degree of the original graph, one might hope that replacing the union bound by the Lov\'asz Local Lemma (LLL) would immediately do the trick. Unfortunately, this is not the case -- in our setting, for graphs of sublogarithmic degree, the LLL only guarantees that the initial decomposition is nearly-regular with probability exponentially small in $n$ (up to logarithmic factors), so that conditioning on near-regularity (the resulting distribution is called the LLL distribution in this case) could change the measure by an exponentially large amount, thereby completely destroying the spreadness. To overcome this, we use the insight that we require the LLL distribution only for some specific events. For these events, we carefully use a refined comparison between the initial distribution and the LLL distribution (\cref{lem:LLL}) to show that spreadness does not degrade much. We believe that this technique of using event-specific comparison bounds for the LLL distribution will be generally useful in the study of thresholds (via constructing spread measures). 

As an interesting and somewhat related aside, we note that the distributions constructed in our proof of \cref{thm:factorization-spread} can be (approximately) sampled from in polynomial time. However, this is only possible because of recent progress in algorithms for sampling from the LLL distribution, in particular \cite{HWY, JPV}. 

\paragraph{\bf Using spread to perform the union bound} In order to find a sufficiently large regular subgraph inside the initial nearly-regular graph, we need to show that the nearly-regular graph has the property of not having `too many' edges between small sets of vertices (\cref{lem:max-flow}). For logarithmic degree, this can be accomplished by a simple union bound over all small sets; however, this fails for sublogarithmic degree, and unfortunately the dependencies are too numerous to employ the LLL. We overcome this issue using the spreadness of the iterates. Roughly, in \cref{prop:inner-iteration-key}, we show using a direct union bound argument that if a regular graph of degree $D$ is sampled from an $O(D/n)$-spread distribution, then with high probability, a random decomposition into nearly-regular subgraphs has the requisite expansion property. 

\paragraph{\bf Lossless edge-absorption} 
In \cite{KKKMO}, the iterative edge-absorption procedure blows up the spread of the initial decomposition into nearly-regular subgraphs by a factor of approximately $2$. While this is certainly sufficient for their result, it is not amenable to recursion, since to go to subgraphs of degree $O(1)$, we must recurse $\Omega(\log^* n)$ times. To circumvent this issue, we devise a simple, yet refined, edge-absorption procedure (\cref{prop:admissibile-decomposition}), which only uses edges from a randomly chosen, polynomially small part of the graph to make corrections. This ensures that the spread of the corrected distribution is approximately the same as the spread of the initial decomposition into nearly-regular subgraphs, thereby permitting recursion.

\subsection{Concurrent and independent work}
Keevash has independently and concurrently proved \cref{thm:threshold-designs} and essentialy the same result as \cref{thm:factorization-spread} (cf.~\cite[Section~2.4]{Kee22} and \cref{def:nice}) in \cite{Kee22}. His proof is substantially different from ours, relying instead on a very careful analysis of a randomized greedy process.

\section{Preliminaries}
\label{sec:prelims}
We will use the following lemma, which follows easily from the max-flow min-cut theorem (see \cite{Lovasz}), and was also employed in \cite{KKKMO}.
\begin{lem}
\label{lem:max-flow}Consider a bipartite graph
$G=(A,B,E)$. Given $f:A\to\mathbb{Z}_{\geq 0}$ and $g:B\to\mathbb{Z}_{\geq 0}$ with $\sum_{a\in A}f(a) = \sum_{b\in B}g(b)$, there
exists a spanning subgraph $H$ of $G$ with $d_{H}(a)=f(a)$ for all $a \in A$ and $d_{H}(b)=g(b)$ for all $b \in B$
if and only if the following holds: for every subset $A'\subseteq A$
and $B'\subseteq B$, $|E_{G}(A',B')|\ge\sum_{a\in A'}f(a)-\sum_{b\notin B'}g(b)$. 
\end{lem}

A key ingredient in our work is the following comparison between the so-called ``Lov\'asz Local Lemma distribution'' and the product distribution on a collection of random variables. 
(see, e.g.,~Theorem 2.1 of \cite{HSS}). The proof of this result follows directly from the inductive proof of the Lov\'asz Local Lemma (see, e.g.~\cite{AS}). 
\begin{definition}
\label{def:dependency-graph}
Given independent random variables $\{X_{i}\}_{i\in I}$
and events ${\mathcal E}_{j}$ ($j\in J$), where each ${\mathcal E}_{j}$
depends on a subset $S_{j} \subseteq I$ of variables, we say that a graph $\Gamma$
on vertex set $J$ is a dependency graph if it holds that $S_{j}\cap S_{j'}=\emptyset$
whenever $j$ and $j'$ are not adjacent in $\Gamma$.
\end{definition}
\begin{proposition}
\label{lem:LLL}
With notation as in \cref{def:dependency-graph}, denote by ${\bf P}$ the usual product measure on the random variables $\{X_i\}_{i \in I}$ and by
$\mathbb{P}$, the conditional measure ${\bf P}(\cdot \mid \cap_{j}{\mathcal E}_j^{c})$. 
Assume that ${\bf P}({\mathcal E}_{j})\le p$ for all
$j \in J$ and that the maximum degree of $\Gamma$, denoted by $\Delta$, satisfies $4p\Delta\le1$.
Given an event ${\mathcal E}$ depending on a subset of variables $S \subseteq I$, letting
$N$ be the number of events ${\mathcal E}_{j}$ $(j\in J)$  with $S_{j}\cap S\ne\emptyset$,
we have 
\[
\mathbb{P}[{\mathcal E}]\le{\bf P}[{\mathcal E}]\exp(6pN).
\]
\end{proposition}

Finally, we will use the following standard version of the Chernoff bound. 
\begin{lem}
\label{lem:Chernoff}Let $X_{1},\dots,X_{m}$ be independent $\operatorname{Bernoulli}(p)$
random variables. Then, for $\delta>0$, 
\begin{align*}
\mathbb{P}[X_{1}+\dots+X_{m}>(1+\delta) pm]
&\le\left(\frac{e^{\delta}}{(1+\delta)^{1+\delta}}\right)^{pm} \leq \exp\left(-\frac{\delta^2 pm}{2+\delta}\right)\\
\mathbb{P}[X_{1}+\dots+X_{m} < (1-\delta) pm]
&\le\left(\frac{e^{-\delta}}{(1-\delta)^{1-\delta}}\right)^{pm} \leq \exp\left(-\frac{\delta^2pm}{2}\right).
\end{align*}
\end{lem}

\section{Proof of \cref{thm:factorization-spread}}
\subsection{Recursion}
Throughout this section, we use the notation appearing in the statement of \cref{thm:factorization-spread}. In the following definition, $S$ denotes a positive integer which, in our application, will be chosen to be sufficiently large depending on $\epsilon$. We define
\[q_{r} := D_{r}^{-1/8}, \quad \delta_{r} := 2^{10}\cdot S\cdot \sum_{r'\leq r}q_{r}.\] 

\begin{definition}[$r$-niceness]
\label{def:nice}
For $r \in \mb{Z}_{\geq 0}$, we say that a bipartite graph $G = (A,B,E)$ with $|A|=|B|=n$ is $r$-nice if the following properties hold: 
\begin{enumerate}[(N1)]
    \item $G$ is $D_G$-regular, with 
    \[D_G = \exp(\pm \delta_{r-1})\cdot D_{r}\]
    \item For all $A'\subseteq A,B'\subseteq B$ with $|B'|\ge|A'| > n/S$ and $n-|B'|=1.01|A'|$,
\[
|E_{G}(A',B\setminus B')|\le \exp(1+ \delta_{r-1}) \cdot D_G|A'|\cdot \max\left(\frac{1}{3},\frac{n-|B'|}{n}\right),
\]
and similarly with the role of $A'$ and $B'$ interchanged.
\end{enumerate}

For $r \in \mb{Z}_{\geq 0}$ and a regular bipartite graph $\mb{G} = (A,B,E)$ with $|A|=|B|=n$ and degree $D_0$, a decomposition $\mc{P}$ of $E(\mb{G})$ is said to be $r$-nice if $|\mc{P}| = D_0/D_{r} (=S^{r})$ and every $G \in \mc{P}$ is $r$-nice.

\end{definition}

\cref{thm:factorization-spread} follows immediately from the following proposition, whose proof is the content of the remainder of this paper. 

\begin{thm}
\label{thm:factorization-spread-precise}
For any $\epsilon \in (0,1/2)$, there exist sufficiently large constants $N_0, S \in \mb{N}$ for which the following holds. Let ${\mathbb G} = (A,B,E)$ be a bipartite graph with $|A|=|B|=n$ which is $D_0$-regular for $D_0 \geq \epsilon n \geq N_0$ and $0$-nice (\cref{def:nice}). For $r \in \mb{N}$, let $D_{r} := D_0/S^{r}$.  
Then, for all $r \in \mb{Z}_{\geq 0}$ such that $D_{r} \geq N_0$, there exists a probability distribution $\bm{\mc{P}}_r$ on decompositions $\mc{P}_r$ of $E(\mb{G})$, supported on $r$-nice decompositions, with spread $2\epsilon^{-1}\cdot D_r/n$. 
\end{thm}

For $r = 0$, note that the distribution supported on the trivial decomposition consisting of only one part satisfies the conclusion of \cref{thm:factorization-spread-precise}. The distributions for $r \geq 1$ will be constructed recursively in \cref{prop:inner-iteration-key}, starting from this choice of $\bm{\mc{P}}_0$.   

\subsection{Admissibility} The recursive construction used to prove \cref{thm:factorization-spread-precise} is based on a procedure to partition a regular bipartite graph $G$ into regular subgraphs. Throughout this subsection, we consider $r \in \mb{Z}_{\geq 0}$ and a $D_G$-regular bipartite graph $G=(A,B,E)$ with $|A|=|B|=n$ and $D_G \geq N_0$, which is $r$-nice (\cref{def:nice}) with respect to the sequence $D_{r}$ in the statement of \cref{thm:factorization-spread-precise}. Recall that $S \in \mb{N}$ will be chosen to be sufficiently large. %

Consider the following collection of random variables:
\begin{itemize}
    \item For each edge $e\in G$, $\pi(e)$ is supported in $[S]( = [s_0 - 1])$.

    \item For each edge $e \in G$, $\xi(e)$ is supported in $\{0,1\}$. 
\end{itemize}

To this collection of random variables, we associate the following collection of subgraphs of $G$:
\begin{itemize}
    \item For $i \in [S]$, $H_i$ denotes the subgraph of $G$ consisting of all edges $e$ for which $\pi(e) = i$. Succinctly, $H_i = \pi^{-1}(i)$.

    \item For $i \in [S]$, $H_i^{+} = H_{i} \cap \{e : \xi(e) = 1\} = H_{i} \cap \xi^{-1}(1)$. %
    \item Let $H^+ = H_1^+\cup \dots \cup H_S^+$. 
\end{itemize}

The next definition collects the properties we will require of these random variables.

\begin{definition}[$r$-admissibility]
\label{def:admissibility}
With notation as above, %
a realisation of the collection of random variables $\{\pi(e), \xi(e)\}_{e \in G}$ is said to be $r$-admissible if the following properties hold:
\begin{enumerate}[(R1)]

    \item for all $v \in V(G)$, for all $i \in [S]$, 
    \[|d_{H_i}(v) - D_G/S| \leq 9\sqrt{\log{D_G}\cdot D_G/S};\]
    
    \item for all $v \in V(G)$, for all $i \in [S]$, 
    \[|d_{H_i^+}(v) - q_rD_G/S| \leq 9\sqrt{\log{D_G}\cdot D_G/S}; \]
    
\end{enumerate}
\begin{enumerate}[(E1)]
    \item for all $i \in [S]$ and all $A'\subseteq A, B'\subseteq B$ with either $|B'|\geq |A'| \geq 4n/5$ or $4n/5 \geq |A'| \geq n/S$ and $n-|B'| < 1.01|A'|$
    \[|E_{H_i^+}(A',B')| \geq (9/10)\cdot |E_G(A',B')|q_r/S,\]
  
    \item for all $i \in [S]$ and all $A'\subseteq A, B'\subseteq B$ with $|A'| \leq n/S$ and $|B'| = 1.01|A'|$ 
    \[|E_{H_i^+}(A',B')| \leq ({D_Gq_r}/{2S})\cdot |A'|,\]
    
    \item for all $i \in [S]$ and all $A'\subseteq A, B'\subseteq B$ with $|B'| \geq |A'| \geq n/S$ and $n-|B'| = 1.01|A'|$,
    \[\left|E_{H_i}(A',B\setminus B')\right| \leq \exp\left(1 + q_{r} + \delta_{r-1}\right)D_G|A'|\beta'/S,\]
    where $\beta' = \max(\kappa, 1-|B'|/n)$,   
    \item for all $i \in [S]$ and all $A'\subseteq A, B'\subseteq B$ with $|B'| \geq |A'| \geq n/S$ and $n-|B'| = 1.01|A'|$,
     \[\left|E_{H_i^+}(A',B\setminus B')\right| \leq 4D_G|A'|\beta'q_{r}/S,
     \]
  where $\beta' = \max(\kappa, 1-|B'|/n)$, 
\end{enumerate}
and similarly the properties $(E1),(E2),(E3),(E4)$ with the roles of $A'$ and $B'$ interchanged. 
\end{definition}

The motivation for the definition of $r$-admissibility comes from the next proposition which shows that given an $r$-admissible realisation, one can decompose an $r$-nice graph $G$ into $(r+1)$-nice regular subgraphs of degree approximately $D_G/S$ in a manner that will turn out to be sufficiently `spread'.  

\begin{proposition}
\label{prop:admissibile-decomposition}
Given any $r$-admissible realisation of $\{\pi(e), \xi(e)\}_{e\in G}$, %
there exists a disjoint collection of regular subgraphs $\{R_i\}_{i \in [S]}$ of $G$ satisfying the following properties:
\begin{enumerate}[(P1)]
    \item $G = R_1 \cup \dots \cup R_{S}$; 
   \item for all $i \in [S-1]$, $H_i \setminus H_i^+ \subseteq R_{i} \subseteq H_i$; %
   \item $H_S \subseteq R_S \subseteq H_S \cup H^+$;
   \item for all $i\in [S]$, the degree of each $R_{i}$ is in $(1\pm 2Sq_{r})D_G/S$; 
\item for all $i \in [S]$, $R_i$ is $(r+1)$-nice i.e.~it satisfies properties $(N1)$ and $(N2)$ in \cref{def:nice} with $D_G := D_{R_i}$ and $D_{r+1}$. 
\end{enumerate}
\end{proposition}

\begin{proof}
Let $K_i := H_i \setminus H_i^+$.
Suppose that for all $i \in [S-1]$, there exists a regular subgraph satisfying 
\begin{equation}
\label{eqn:regular-subgraph}
K_i \subseteq R_{i} \subseteq  H_i.
\end{equation}
Then, we set $R_{S} := G \setminus (R_1 \cup \dots \cup R_{S-1})$. Note that since $G$ is regular and $R_1,\dots, R_{S-1}$ are disjoint regular subgraphs of $G$, $R_S$ is also regular.

We claim that for any $r$-admissible realisation, we can always find regular subgraphs satisfying \cref{eqn:regular-subgraph}. Before proving this claim, let us verify that the construction thus obtained satisfies the desired properties. $(P1), (P2), (P3)$ are satisfied by construction. 
It remains to verify that $(P4)$ and $(P5)$ hold. For $(P4)$, we have for all $i \in [S]$ that
\[\delta_{H_i} - \Delta_{H_i^+} \leq d_{R_i} \leq \Delta_{H_i} + \Delta_{H^+},\]
where $d_{G'},\Delta_{G'},\delta_{G'}$ denote respectively the average, maximum and minimum degree of a subgraph $G'$ of $G$. Using $(R1), (R2)$, we get that
\begin{align*}
\underbrace{\frac{D_G}{S} - 18\sqrt{\log{D_G}\cdot \frac{D_G}{S}} - q_{r}\frac{D_G}{S}}_{\geq \frac{D_G}{S} - 4q_{r}\frac{D_G}{S}}\leq d_{R_i} &\leq \frac{D_G}{S} + 9S\sqrt{\log{D_G}\cdot \frac{D_G}{S}} + q_r D_G \\
&\le \frac{D_G}{S} + 2Sq_r \frac{D_G}{S},
\end{align*}
where we have used that $q_{r} = D_{r}^{-1/8}$ and $D_{r} \geq N_0$, a sufficiently large absolute constant. In particular, using that $D_G = \exp(\pm \delta_{r-1})D_{r}$, we get that $D_{R_i} = \exp(\pm \delta_{r})D_{r+1}$,
thereby verifying $(N1)$ in $(P5)$. For $(N2)$ in $(P5)$, note that by $(E3),(E4)$, for all $A'\subseteq A, B'\subseteq B$ with $|B'| \geq |A'| > n/S$, $n-|B'| = 1.01|A'|$, and $\beta' = \max(1/3, 1-|B'|/n)$, we have that 
\begin{align*}
|E_{R_i}(A',B\setminus B')|
&\le |E_{H_i}(A',B\setminus B')| + \sum_{i \in [S]}|E_{H_i^+}(A',B\setminus B')|\\
&\le \left(\exp(1+ q_r + \delta_{r-1}) + 4Sq_r\right)\frac{D_G}{S}|A'|\beta'\\
&\le \exp(1+q_r + \delta_{r-1})\left(1+4Sq_r\right)^{2}{D_{R_i}}|A'|\beta'\\
&\le \exp(1+\delta_r){D_{R_i}}|A'|\beta',
\end{align*}
where we have used that $q_{r} = D_{r}^{-1/8}$ and $D_{r} \geq N_0$, a sufficiently large absolute constant.  

Finally, we establish the existence of regular subgraphs satisfying \cref{eqn:regular-subgraph} by applying \cref{lem:max-flow} on the graph $H_{i}^+ = (A,B, E(H_i^+))$
with $f(a)=d-d_{K_{i}}(a)$ and $g(b)=d-d_{K_{i}}(b)$
for $d:=d_{K_{i}}+10^5\sqrt{(D_G/S)\log D_G}$, where for notational convenience, we assume that second summand is an integer. Observe that this choice of $f$ and $g$ is valid. Indeed, since $|A| = |B|$, we have that $\sum_{a \in A}f(a) = \sum_{b\in B}g(b)$ and moreover, by $(R1), (R2)$, 
\begin{align*}
d-d_{K_{i}}(v) 
\ge d_{K_i} + 10^5\sqrt{(D_G/S)\log D_G} -\Delta_{K_i} 
 > (10^5-72)\sqrt{(D_G/S)\log D_G},
\end{align*}
so that $f,g \geq 0$. For later use, note also that by a similar computation,
\[
d-d_{K_{i}}(v)\le d_{K_i}+ 10^5\sqrt{(D_G/S)\log D_G} - \delta_{K_i} < (10^5+72)\sqrt{(D_G/S)\log D_G}.
\]
It remains to show that for all $A'\subseteq A$ and $B'\subseteq B$,
\begin{equation}
\left|E_{H_{i}^+}(A',B')\right|\ge \Delta(A',B'):=\sum_{a\in A'}(d-d_{K_{i}}(a))-\sum_{b\notin B'}(d-d_{K_{i}}(b)).\label{eq:A'B'cond}
\end{equation}
Since
\[
\sum_{a\in A'}(d-d_{K_{i}}(a))-\sum_{b\notin B'}(d-d_{K_{i}}(b))=\sum_{b\in B'}(d-d_{K_{i}}(b))-\sum_{a\notin A'}(d-d_{K_{i}}(a)),
\]
it suffices by symmetry to assume that $|B'|\ge|A'|$. 

We consider a few different cases. In each case, we will use that $D_r \geq N_0$, a sufficiently large absolute constant. 

\textbf{Case 1:} $|A'| \ge 4n/5$. In this case, $|B'|\ge|A'| \ge 4n/5$, so that by $(E1)$,
\begin{align*}
\left|E_{H_{i}^+}(A',B')\right|
&\geq (9/10)\cdot \left|E_G(A',B')\right|q_r/S
\geq (9/10)\cdot (D_G|A'| - D_G(n-|B'|))q_{r}/S \geq 2D_{G}q_{r}n/(5S),
\end{align*}
and hence
\begin{align*}
\Delta(A',B') 
&\leq (10^5+72)|A'|\sqrt{(D_G/S)\log D_G}-(10^5-72)(n-|B'|)\sqrt{(D_G/S)\log D_G}\\
&\leq (10^{5} + 72)n\sqrt{(D_G/S)\log D_G} \leq 2D_{G}q_{r}n/(5S) \leq \left|E_{H_{i}^+}(A',B')\right|.
\end{align*}

\textbf{Case 2:} $n/S\le|A'|<4n/5$. We may assume that $n-|B'| < 1.01|A'|$, since otherwise,
$$\Delta(A',B') \leq \left((10^5+72)|A'| - (10^5 - 72)(n-|B'|)\right)\cdot \sqrt{(D_G/S)\log{D_G}} < 0$$
and \cref{eq:A'B'cond} trivially holds. Assuming that $n-|B'| < 1.01|A'|$, we have by $(E1)$ and the $r$-niceness of $G$ that
\begin{align*}
\left|E_{H_{i}^+}(A',B')\right|
&\geq (9/10)\cdot \left|E_G(A',B')\right|q_r/S \geq (9/10)\cdot q_{r}/S\cdot \left(D_G |A'| - |E_G(A', B\setminus B')|\right)\\
&\geq D_{G}|A'|q_{r}/(100S)  \geq (10^{5}+72)|A'|\sqrt{(D_G/S)\log{D_G}} \geq \Delta(A',B').
\end{align*}

\textbf{Case 3: }$|A'|<n/S$. As before, we may assume that $n-|B'| < 1.01|A'|$. 
By (R2) and (E2), we have
\begin{align*}
\left|E_{H_{i}^+}(A',B')\right|
&\geq \delta_{H_i^+}|A'| - |E_{H_i}^+(A',B\setminus B')| \geq \left(\frac{D_Gq_r}{S} - 9\sqrt{(D_G/S)\log D_G}- \frac{D_Gq_r}{2S}\right)|A'|\\
&\geq \frac{D_Gq_r}{4S}|A'| \geq (10^5+72)|A'|\sqrt{(D_G/S)\log D_G} \geq \Delta(A',B'). \qedhere
\end{align*}

\begin{remark}
\label{rmk:edge-lb}
    Our calculation in Case 1 and Case 2 in the above proof shows that for all $A'\subseteq A$ and $B'\subseteq B$ satisfying either $|B'|\geq |A'| \geq 4n/5$ or $4n/5 \geq |A'| \geq n/S$ and $n-|B'| < 1.01|A'|$, we have
    \[|E_G(A', B')| \geq D_G|A'|/100,\]
    and similarly with the roles of $A'$ and $B'$ interchanged. 
\end{remark}

\end{proof}

\subsection{The building block}
Having established \cref{prop:admissibile-decomposition}, our goal now is to construct a distribution on $\{\pi(e),\xi(e)\}_{e\in G}$ such that for $r$-nice graphs $G$, the random variables are $r$-admissible with sufficiently high probability and moreover, if $G$ is drawn from an $O(D_r/n)$-spread distribution, then the resulting distribution on $(r+1)$-nice graphs is $O(D_{r+1}/n)$-spread. As in the previous subsection, we consider $r \in \mb{Z}_{\geq 0}$ and a $D_G$-regular bipartite graph $G=(A,B,E)$ with $|A|=|B|=n$ and $D_G \geq N_0$, which is $r$-nice (\cref{def:nice}) with respect to the sequence $D_{r}$ in the statement of \cref{thm:factorization-spread-precise}. For $v \in A \cup B$, consider the following events:
\begin{itemize}
\item $\mc{R}_1(v)$ denotes the event that for some $i \in [S]$, $v$ does not satisfy $(R1)$ in \cref{def:admissibility}.
\item $\mc{R}_2(v)$ denotes the event that for some $i \in [S]$, $v$ does not satisfy $(R2)$ in \cref{def:admissibility}.
\end{itemize}
Let $\bm{P}_{G}$ denote the product distribution on $\{\pi(e), \xi(e)\}_{e\in G}$, where each $\pi(e)$ is distributed uniformly in $[S]$ and each $\xi(e) \sim \on{Bernoulli}(q_r)$. We define the probability distribution $\mb{P}_G$ on $\{\pi(e),\xi(e)\}_{e\in G}$ to be the conditional distribution
\[\mb{P}_{G} := \bm{P}_{G}\left[\cdot \mid (\cap_{v \in V(G)}\mc{R}_1(v)^{c}) \bigcap (\cap_{v \in V(G)}\mc{R}_2(v)^{c})\right].\]

Using the Chernoff bound (\cref{lem:Chernoff}) and \cref{lem:LLL}, we show that with very high probability, a sample from $\mb{P}_G$ satisfies all properties of admissibility except possibly $(E2)$.

\begin{lem}
\label{lem:admissibility-easy}
Let $r$ and $G$ be as above. With probability at least $1 - \exp(-n)$, a random realisation of $\{\pi(e), \xi(e)\}_{e\in G}$ drawn from the distribution $\mb{P}_G$ satisfies all properties in \cref{def:admissibility}, except possibly $(E2)$. %
\end{lem}
\begin{proof}
    Note that $(R1)$ and $(R2)$ are always satisfied by construction. 
In order to control the probability that at least one of $(E1),(E3),(E4)$ fails, we will use \cref{lem:LLL}. To this end, we begin by observing that a direct application of the Chernoff bound (\cref{lem:Chernoff}) and union bound shows that for all $v \in V(G)$ and $j \in [\gamma]$, 
\begin{align}
\label{eqn:chernoff}
    \max\left(\bm{P}_{G}[\mc{R}_1(v)], \bm{P}_{G}[\mc{R}_2(v)]\right) &\leq D_{r}^{-20};
\end{align}

For the collection of events $\{\mc{R}_1(v), \mc{R}_2(v)\}_{v\in V(G)}$ and the product measure $\bm{P}_{G}$ on $\{\pi(e), \xi(e)\}_{e\in G}$, it is easily verified that a dependency graph $\Gamma$ (in the sense of \cref{def:dependency-graph}) is given by $G$ itself, so that the condition in \cref{lem:LLL} is satisfied. By the Chernoff bound (\cref{lem:Chernoff}) and \cref{rmk:edge-lb}, for any $A'\subseteq A, B'\subseteq B$ satisfying either $|B'|\geq |A'| \geq 4n/5$ or $4n/5 \geq |A'| \geq n/S$ and $n-|B'| < 1.01|A'|$, we have for any $i \in [S]$ that
\begin{align*}
\bm{P}_{G}\left[\underbrace{\left|E_{H_i^+}(A',B')\right| < (9/10)\cdot E_G(A',B')q_r/S}_{\mathcal{E}_{1}(A',B',i)}\right] 
&\leq \exp\left(-\frac{E_G(A',B')q_r}{200S}\right) \leq \exp\left(-\frac{D_G|A'|q_r}{20000S}\right).
\end{align*}
Since there are $4n$ events of the form $\mathcal{R}_1(v)$ or $\mathcal{R}_2(v)$, it follows from \cref{lem:LLL} that
\begin{align*}
\mb{P}_{G}'[\mathcal{E}_1(A',B',i)] \leq \exp\left(-\frac{D_G|A'|q_r}{20000S}\right)\exp\left(24nD_{r}^{-20}\right) &\leq \exp\left(-\frac{D_G|A'|q_r}{40000S}\right) \leq \exp\left(-D_r^{3/4}|A'|\right),
\end{align*}
so that a union bound over the relevant choices of $A',B'$ and $i$ (together with the case where the roles of $A'$ and $B'$ are interchanged) shows that
\begin{align*}
    \mb{P}_{G}[(E1)\text{ fails}] \leq \exp(-nD_{r}^{1/2}). 
\end{align*}

Since $G$ is $r$-nice, we have by $(N2)$ that for all $A'\subseteq A,B'\subseteq B$ with $|B'|\ge|A'| > n/S$ and $n-|B'|=1.01|A'|$,
\[
|E_{G}(A',B\setminus B')|\le \exp(1+\delta_{r-1})D_G|A'|\underbrace{\max\left(\frac{1}{3},\frac{n-|B'|}{n}\right)}_{\beta'}.
\]
Therefore, by the Chernoff bound (\cref{lem:Chernoff}), 
\begin{align*}
    \bm{P}_{G}\left[|E_{H_i}(A',B\setminus B')| > \exp(1+10q_r + \delta_{r-1})D_G|A'|\beta'/S\right] &\leq \exp\left(-\frac{q_r^{2} D_G |A'|\beta'}{S}\right) 
    \leq \exp\left(-D_{r}^{5/8}|A'|\right),
\end{align*}
so once again, using \cref{lem:LLL} and the union bound over relevant choices of $A',B',i$ (together with the case where the roles of $A'$ and $B'$ are interchanged) shows that
\[\mb{P}_{G}[(E3)\text{ fails}] \leq \exp(-n D_{r}^{1/2}).\]
A similar computation shows that 
\[\mb{P}_G[(E4)\text{ fails}] \leq \exp(-nD_r^{1/2}). \qedhere\]

\end{proof}

It remains to show that the probability that $(E2)$ is violated by a sample from $\mb{P}_G$ is sufficiently small. While this is not necessarily true for all $r$-nice graphs $G$, we show in the key \cref{prop:inner-iteration-key} that if $G$ is drawn from an $O(D_r/n)$-spread distribution over $r$-nice graphs, then with very high probability over the choice of $G$, $\mb{P}_G$ does have this property.

\begin{definition}[r-excellence]
\label{def:excellent}
For $r \in \mb{Z}_{\geq 0}$, we say that a  bipartite graph $G = (A,B,E)$ with $|A|=|B|=n$ is $r$-excellent if $G$ is $r$-nice and moreover, the probability that a random realisation $\{\pi(e), \xi(e)\}_{e\in G}$ sampled from $\mb{P}_G$ is $r$-admissible is at least $1-n^{-50}$.  

For $r \in \mb{Z}_{\geq 0}$ and a regular bipartite graph $\mb{G} = (A,B,E)$ with $|A| = |B| = n$ and degree $D_0$, we say that a decomposition $\mc{P}$ of $E(\mb{G})$ is $r$-excellent if $|\mc{P}| = D_0/D_r (=S^r)$ and every $G \in \mc{P}$ is $r$-excellent.    
\end{definition}

\begin{proposition}
\label{prop:inner-iteration-key}
For any $C > 0$, there exist $N_0$ and $S$ such that the following holds. With notation as above, let $\bm{G}_{r}$ be a probability distribution supported on $r$-nice bipartite graphs with vertex sets $(A,B)$ of size $|A| = |B| = n$. If $\bm{G}_{r}$ is $C\cdot D_{r}/n$-spread, where $D_r \geq N_0$, then the probability that $G \sim \bm{G}_{r}$ is $r$-excellent is at least $1-n^{-50}$.  
\end{proposition}

\begin{proof}
By \cref{lem:admissibility-easy}, it suffices to show that with probability at least $1-n^{-200}$, the collection $\{\pi(e), \xi(e)\}_{e\in G}$ sampled from the measure $\bm{G}_{r} \times \mb{P}_G$ (i.e.~first sample $G$ from $\bm{G}_r$ and then sample $\{\pi(e), \xi(e)\}_{e\in G}$ from $\mb{P}_G$) satisfies $(E2)$. If $(E2)$ fails, then there exists $i \in [S]$ and $A'\subseteq A, B'\subseteq B$ with $|A'| \leq n/S$ and $|B'| = 1.01|A'|$ (or $A',B'$ with their roles switched) such that
\begin{align*}
\left|E_{H_i^+}(A',B')\right| > \frac{D_Gq_r}{2S}|A'| > \frac{D_rq_r}{4S}|A'|.
\end{align*}
Let $\mc{T}$ denote the collection of all subgraphs $T = (A',B',E(T))$ of $K_{n,n}$ with $|A'| \leq n/S$, $|B'| = 1.01|A'|$, and $|E(T)| = D_rq_r|A'|/4S$. For fixed $i \in [S]$ and any $T \in \mc{T}$, note that given an $r$-nice $G \supseteq T$, the event $T\subseteq H_i^{+}$ shares variables with at most $4|E(T)|$ events of the form $\mc{R}_1(v), \mc{R}_2(v)$. Therefore, 
using the spreadness of $\bm{G}_{r}$, \cref{eqn:chernoff}, and \cref{lem:LLL}, we have that
\begin{align*}
    \left(\bm{G}_r \times \mb{P}_G\right) [T \subseteq H_i^+]
    &= \bm{G}_r[T \subseteq G]\cdot \mb{P}_G[T \subseteq H_i^+ \mid T \subseteq G]\\
    &\leq \left(\frac{CD_{r}}{n}\right)^{|E(T)|}\cdot \bm{P}_{G}[T\subseteq H_i^+ \mid T \subseteq G]\cdot e^{4|E(T)|D_{r}^{-20}} \\
    &\leq \left(\frac{e^{4D_{r}^{-20}}CD_r}{n}\right)^{|E(T)|}\cdot \left(\frac{q_r}{S}\right)^{|E(T)|} \leq \left(\frac{2Cq_rD_r}{nS}\right)^{|E(T)|}.
\end{align*}
Therefore, by the union bound, and assuming that $D_r\ge N_0\ge 2^{20}S^2$ and $S$ is sufficiently large compared to $C$, it follows that
\begin{align*}
     \left(\bm{G}_r \times \mb{P}_G\right) \left[\bigcup_{T \in \mc{T}}\{T \subseteq H_i^+\}\right]
     &\leq \sum_{T\in \mc{T}}\left(\frac{2Cq_rD_r}{nS}\right)^{|E(T)|} 
\leq \sum_{k=1}^{n/S}\binom{n}{k}\binom{n}{1.01k}\binom{1.01k^{2}}{D_{r}q_rk/4S}\cdot \left(\frac{2Cq_rD_{r}}{nS}\right)^{D_{r}q_rk/4S}\\
     &\leq \sum_{k=1}^{n/S}\left(\frac{4n}{k}\right)^{2k}\cdot \left(\frac{100Ck}{n}\right)^{D_{r}q_rk/4S} \leq \sum_{k=1}^{n/S}\left(\frac{100Ck}{n}\right)^{D_{r}q_rk/8S} \leq n^{-400}. \qedhere
\end{align*}
\end{proof}
\subsection{Putting everything together} 
Finally, we combine everything to show how to construct an $O(D_{r+1}/n)$-spread distribution on $(r+1)$-nice decompositions, starting with an $O(D_{r}/n)$-spread distribution on $r$-nice decompositions. As discussed earlier, this completes the proof of \cref{thm:factorization-spread-precise}. Inductively, the following proposition establishes a probability distribution $\bm{\mc{P}}_{r}$ on decompositions $\mc{P}_r$ of $E(\mb{G})$ which is $\epsilon^{-1}\exp(\sum_{r'< r}5Sq_{r'})D_r/n$-spread. Note that, for $N_0$ sufficiently large, $\sum_{r' < r}q_r < 1/(100S)$ whenever $D_r \ge N_0 \ge 2^{20}S^{10}$ and thus $\bm{\mc{P}}_{r}$ is $2\epsilon^{-1}D_r/n$-spread for any such $r$. 
\begin{proposition}
\label{prop:spread-partitions}
With notation as in the statement of \cref{thm:factorization-spread-precise}, let $\bm{\mc{P}}_r$ denote a probability distribution supported on $r$-nice decompositions of $E(\mb{G})$, and let $\bm{\mc{P}}_{r+1}$ be the distribution on $(r+1)$-nice decompositions of $E(\mb{G})$ defined as follows: first sample $\mc{P}$ from the conditional distribution $\bm{\mc{P}}_{r} \mid r\on{-excellent}$, then sample 
$\{\pi(e),\xi(e)\}_{e\in G}$ from the conditional distribution $\mb{P}_G \mid r\on{-admissible}$ for each $G \in \mc{P}$, and finally use the procedure in \cref{prop:admissibile-decomposition}. For any $C \leq 2\epsilon^{-1}$, if $\bm{\mc{P}}_r$ is $C\cdot D_{r}/n$-spread, then $\bm{\mc{P}}_{r+1}$ is  $C\cdot e^{5Sq_r}\cdot D_{r+1}/n$-spread. 
\end{proposition}

\begin{proof}
Let $\bm{\mc{P}}_r=(\bm{G}_1,\dots,\bm{G}_{S^r})$. 
$
D_{\bm{G}_i}=\exp(\pm \delta_{r-1})\cdot D_r. 
$
Since $\bm{\mc{P}}_r$ is $C\cdot D_r/n$-spread, the same holds by marginalization for each of the distributions $\bm{G}_1,\dots, \bm{G}_{S^r}$.  %
Therefore, by \cref{prop:inner-iteration-key} and the union bound, the probability that $\mc{P}_r$ drawn from $\bm{\mc{P}}_r$ is $r$-excellent is at least $1-n^{-49}$; in particular, the conditional distribution $\wt{\bm{\mc{P}}}_r := \bm{\mc{P}}_{r} \mid r\text{-excellent}$ is spread with parameter at most $(1-n^{-49})^{-1}\cdot C\cdot D_r/n \leq e^{q_r}CD_r/n$.

Recall that $\bm{\mc{P}}_{r+1} = (\bm{G}_{i,j})_{i \in [S^r], j \in [S]}$ is obtained by first sampling $(\bm{G}_{1},\dots,\bm{G}_{S^r})$ from $\wt{\bm{\mc{P}}}_r$ and then decomposing $\bm{G}_i$ into $\bm{G}_{i,1}\cup \dots \cup \bm{G}_{i,S}$, independently for each $i \in [S^r]$, by sampling $\{\pi(e),\xi(e)\}_{e\in \bm{G}_i}$ from $\wt{\mb{P}}_{\bm{G}_i} := \mb{P}_{\bm{G}_i} \mid r\on{-admissible}$ and using the procedure in \cref{prop:admissibile-decomposition}. Let $\{T_{i,j}\}_{i \in [S^r], j \in [S]}$ be disjoint subsets of $E(\mb{G})$. For $i \in [S^r]$, let $T_i = \bigcup_{j\in [S]}T_{i,j}$. Let $I \subseteq [S^r]$ denote the subset of indices $i$ for which $T_i \neq \emptyset$. Then,
\begin{align}
\label{eqn:combine}
    \Pr\left[\bigcap_{i \in [S^r], j \in [S]}\{T_{i,j}\subseteq \bm{G}_{i,j}\} \right]
    &=  \Pr\left[\bigcap_{i \in I, j \in [S]}\{T_{i,j}\subseteq \bm{G}_{i,j}\} \mid \bigcap_{i \in I}\{T_i \subseteq \bm{G}_i\}\right]\cdot \Pr\left[\bigcap_{i \in I}\{T_i \subseteq \bm{G}_i\}\right].
\end{align}

For the second term in the product, using the spreadness of $\wt{\bm{\mc{P}}}_r$, we have
\begin{equation}
\label{eq:bound1}
 \Pr\left[\bigcap_{i \in I}\{T_i \subseteq \bm{G}_i\}\right] \leq \left(\frac{e^{q_r}CD_r}{n}\right)^{\sum_i |T_i|}.
\end{equation}
For the first term in the product, we have the upper bound
\begin{align}
\label{eq:bound2}
    \prod_{i \in I}\wt{\mb{P}}_{\bm{G}_i}[\cap_{j \in [S]}T_{i,j} \subseteq \bm{G}_{i,j} \mid T_i \subseteq \bm{G}_i]
    &\leq (1-n^{-50})^{-|I|}\cdot \prod_{i\in I}{\mb{P}}_{\bm{G}_i}[\cap_{j \in [S]}T_{i,j} \subseteq \bm{G}_{i,j} \mid T_i \subseteq \bm{G}_i],
\end{align}
where we have used that each $\bm{G}_i$ is supported on $r$-excellent graphs. Recalling that for $j \in [S]$, $\bm{G}_{i,j}$ is always contained in  $H_j \cup H^+$, we can upper bound the $i^{th}$ term in the product by
\begin{align}
\label{eq:bound3}
    \mb{P}_{\bm{G}_i}[\cap_{j \in [S]}\cap_{e \in T_{i,j}} \{(\pi(e) = j) \cup (\xi(e) = 1)\} \mid T_i \subseteq \bm{G}_i] \nonumber\\
    \leq e^{20|T_i|D_r^{-20}}\left(\frac{1}{S} + q_r\right)^{|T_i|} 
    \leq \left(\frac{e^{3Sq_r}}{S}\right)^{|T_i|},
\end{align}
where in the first inequality, we use \cref{lem:LLL} and \cref{eqn:chernoff}. Combining \cref{eqn:combine,eq:bound1,eq:bound2,eq:bound3}, we have that
\begin{align*}
    \Pr\left[\bigcap_{i \in [S^r], j \in [S]}\{T_{i,j}\subseteq \bm{G}_{i,j}\} \right]
    &\leq \left(\frac{e^{q_r}CD_r}{n}\right)^{\sum_{i}|T_i|}\cdot \left(\frac{e^{3Sq_r}}{S}\right)^{\sum_i|T_i|}\cdot (1-n^{-50})^{-|I|}\\
    &\leq \left(\frac{e^{5Sq_r}CD_r}{nS}\right)^{\sum_{i}|T_i|} \leq  \left(\frac{e^{5Sq_r}CD_{r+1}}{n}\right)^{\sum_{i,j}|T_{i,j}|},
\end{align*}
as desired. \qedhere

\section*{Acknowledgements.}
The second author would like to thank David Conlon and Jacob Fox for helpful discussions. The second author is supported by a Two Sigma Fellowship.

\end{proof}

\

\end{document}